\documentclass[reqno]{amsart}
\usepackage{amsmath, amsfonts, amssymb, latexsym}
\usepackage[dvips]{graphicx}
\graphicspath{{Img/}}

\newtheorem{theorem}{Theorem}

\theoremstyle{plain}

\newtheorem{corollary}{Corollary}
\newtheorem{definition}{Definition}

\newtheorem{problem}{Problem}[]
\newtheorem{proposition}{Proposition}
\newtheorem{remark}{Remark}

\newcommand{\eq}{\hspace*{-2mm}&=&\hspace*{-2mm}}
\newcommand\const{\operatorname{const}}

\begin{document}

\title[The plasticity of convex sets]
{The plasticity of non-overlapping convex sets\\ in
$\mathbb{R}^2$}

\author{Anastasios N. Zachos}

\address{University of Patras, Department of Mathematics,, GR-26500 Rion, Greece}

\email{azachos@gmail.com}


\begin{abstract}
We study a generalization of the weighted Fermat-Torricelli
problem in the plane, which is derived by replacing vertices of
a~convex polygon by 'small' closed convex curves with weights
being positive real numbers on the curves, we also study its
genera\-lized inverse problem. Our solution of the problems is
based on the first variation formula of the length of line
segments that connect the weighted Fermat-Torricelli point with
its projections onto given closed convex curves. We find the
'plasticity' solutions for non-overlapping circles with variable
radius.
\end{abstract}

\maketitle

\section{Introduction}
\label{sec:intro}

The extremum problem
formulated by Fermat and after a few years solved by Torricelli,
is as follows: Given three points $A, B$ and $C$ in the Euclidean
plane $\mathbb{R}^2$, the task is to find a point $P$ such that
the sum of distances $PA + PB + PC$ is minimal. The {weighted
\textbf{Fermat-Torricelli}} (\textbf{F-T}) problem is to find the
(unique) point that minimizes the sum of the weighted distances
(i.e., multiplied by positive numbers -- weights) from three given
points in $\mathbb{R}^2$.
The~study of the weighted F-T problem in the plane
and its inverse
is given in \cite{Gue/Tes:02,Zach/Zou:08}. For historical remarks
and generalizations (e.g. in Banach spaces) of the weighted F-T
problem, the reader can consult
\cite{BolMa/So:99,Kup/Mar:97,Mord:11,Mord:12,Mord:12b}.

In~the paper, we study the {generalized F-T} problem for $n\ge3$
convex sets in $\mathbb{R}^2$ and provide a method of its study.
We are based on a technique of differentiation of the length of
geodesics on a $C^{2}$-surface with respect to arc length, see
e.g. \cite{VToponogov:05}, and applying the parametrization method
of~\cite{Zachos/Cots:10,CotsiolisZachos:11}.

Let $A_{1}A_{2}\ldots A_{n}\ (n\ge3)$ be a convex polygon in
$\mathbb{R}^2$ and $\gamma_i$ a convex curve surrounding $A_{i}$
that meets orthogonally at points $D_{ij}\ (j\ne i)$ two sides
containing~$A_{i}$, and the curves do not intersect inside the
polygon, see Fig.~\ref{fig1} for $n=3$.
Segments on the edges of the polygon and $n$ arcs of $\gamma_i$
bound a curvilinear $2n$-gon~$\Omega$. For $P\in\Omega$, let
$A_i'=\pi_i(P)$ be projections of $P$ onto $\gamma_i$. Note that
the segments $PA_i'$ intersect $\gamma_i$ orthogonally.
For $n=3$, let $\varphi_{Q}$ be the angle between the line
segments ${RP}$ and ${SP}$ for $Q,R,S\in\{A_1',A_2',A_3'\}$ and
$Q\ne R\ne S$.

\begin{problem}[The generalized F-T problem in $\mathbb{R}^{2}$]\label{probr2}\rm
Find a point $P\in\mathbb{R}^2$ such that
\begin{equation}\label{minimum}
 \sum\nolimits_{i=1}^{n} w_{i}\,{d}(P, \gamma_i) \to \min\,,
\end{equation}
where $w_{1},\ldots,w_{n}$ are given positive numbers (weights)
and $d$ -- the distance.
\end{problem}

\begin{figure}
\centering
\includegraphics[scale=0.35]{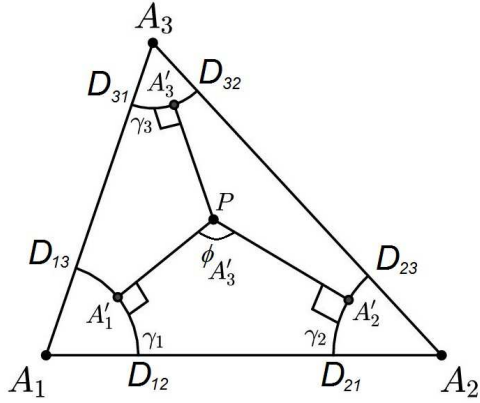}
\caption{The generalized F-T problem for three closed convex
curves in $\mathbb{R}^{2}$.} \label{fig1}
\end{figure}

\begin{problem}[The inverse generalized F-T problem in $\mathbb{R}^{2}$]\label{Problem-2}
Given a generalized F-T point $P$ inside $\Omega$ with the
vertices defined by orthogonal projections of $P$ onto
$\gamma_{i}\ (i=1,\dots,n)$
find positive weights $w_{i}\in\mathbb{R}$ such that
\begin{equation*}
 \sum\nolimits_{i=1}^{n} w_{i} = 1.
\end{equation*}
\end{problem}

\smallskip
The paper is organized as follows. Sect.~\ref{sec:2R} is devoted
to solution of Problem~\ref{probr2}.


In~Sects.~\ref{sec:inverse} and \ref{sec:non-over} we characterize
solutions of Problem~\ref{Problem-2}.

\section{The generalized F-T problem for closed convex sets in $\mathbb{R}^2$}
\label{sec:2R}

In the section we assume that the weights $w_i$ are positive real
numbers which correspond on convex curves $\gamma_i$, and
generalize results of \cite{Zach/Zou:08}, where $\gamma_i=A_i$.


\begin{proposition}\label{L-01}
The function $f=\sum\nolimits_{i} w_{i}\,{d}(P, \pi_i{P})$
is convex in $\Omega$.
\end{proposition}

\textbf{Proof}. It is sufficient to show that the function
$f_i=w_{i}\,{d}(P, \pi_i(P))$ is strictly convex in $\Omega$. We
shall compute the second differential of $f_i$ at a point $P$ in
any direction $X$. Let $P_s$ be a smooth curve with the properties
$P_0=P$ and $\frac{d}{ds}\,P_s=\alpha X$ (for some $\alpha\ne0$)
such that $s$ is the natural parameter of $\pi_i(P_s)$ (the part
of $\gamma_i$). Denote by $f_i(s)=w_i(s)\,l_i(s)$, where
$l_i(s)={d}(P_s, \pi_i({P_s}))\ge0$. It is known that $l_i'(0)=0$
and $l_i''(0)>0$ (see in
\cite{Ivan/Tuzh:01},\cite[Proposition~6.1,
Corollary~6.1]{Ivan/Tuzh:94}). By conditions, we have
\begin{equation*}
f_i'' = w_i\,l_i''>0\quad
 \mbox{ at }\quad s=0.
\end{equation*}
Hence $d^2 f_i(X,X)>0$. The sum $f=\sum_i f_i$ of strictly convex
functions is also strictly convex. \qed

\begin{theorem}\label{Th1}
 The solution $(P_F)$  of Problem~\ref{probr2} exists and is unique.
\end{theorem}


\textbf{Proof}. By Proposition~\ref{L-01}, the objective function
in (\ref{minimum}) is strictly convex; hence,
it has one minimum point on~$\Omega$, see \cite{Bor/Van:10},
\cite[p.~263]{Rockafellar:97}. \qed

\smallskip

The next theorem and corollary can be easily extended for any
$n>3$.


\begin{theorem}\label{T-sol2}
If the generalized F-T point $P$ is an interior point of $\Omega$
$($Fig.~\ref{fig1}$)$ then each angle $\varphi_{i}$ can be
expressed as a function of $w_{i},\ (i=1,2,3)$, as
\begin{equation}\label{E-wi}
\begin{array}{ccc}
 \cos\varphi_{1}= \frac{w_{1}^{2}-w_{2}^{2}-w_{3}^{2}}{2\,w_{2} w_{3}},\quad
 \cos\varphi_{2} = \frac{w_{2}^{2}-w_{1}^{2}-w_{3}^{2}}{2\,w_{1} w_{3}},\quad
 \cos\varphi_{3} = \frac{w_{3}^{2}-w_{1}^{2}-w_{2}^{2}}{2\,w_{1}
 w_{2}}.
\end{array}
\end{equation}

\end{theorem}

\textbf{Proof}.



 Let
$\gamma_{i}:\vec{r}_{i}(s)\equiv(x_{i}(s),y_{i}(s),z_{i}(s))$ be
three convex curves in $\mathbb{R}^{2}$,
$\gamma_{Pi}:\vec{r}_{Pi}(s)\equiv(x_{Pi}(s),y_{Pi}(s),z_{Pi}(s))$
be three smooth curves in $\mathbb{R}^{2}$ issued from the point
$P$,
$l_i(s)\equiv\sqrt{(x_{i}(s)-x_{Pi}(s))^2+(y_{i}(s)-y_{Pi}(s))^2+(z_{i}(s)-z_{Pi}(s))^2}$
the length of the line segment $\gamma_{i}(s)\gamma_{Pi}(s)$,
$d(P,\gamma_{i})\equiv l_{Q}$ for
$Q\in\{A_1^{\prime},A_2^{\prime},A_3^{\prime}\}$ the distance.
Let us join $Q$ and $P$ by the line segment defined by the line
$c_{Q}(s,t)=c_{QP}(t)$, where $t\in[0,1]$ is a canonical parameter
such that $c_{QP}(0)=Q$ and $c_{QP}(1)=P$. The~function
\[
 l_{Q}(P)=l_{Q}(c_{QP}(t))=l_{Q}(c_{Q}(s,t))=l_{Q}(s)
\]
is differentiable with respect to~$s$. Denote by
$\varphi^i_{c}(s)$ the angle between
$-\frac{dc}{ds}\equiv{\overrightarrow{\gamma_{i}(s)\gamma_{Pi}(s)}}$
and $\frac{d}{ds}\,\vec{r}_{Pi}$ at $P$ and by $\alpha_{i}(s)$ the
angle between $\frac{dc}{ds}$ and $\frac{d}{ds}\,\vec{r}_{i}$
at~$Q$. We~will prove that $\alpha_{i}(s)=\frac{\pi}{2}$ for
$i=1,2,3$.

If any of these angles is not $\frac{\pi}{2}$, say, for $A_{1}'$,
then we move $A_{1}'\to A_{1}''$ along $\gamma_{1}$ and the length
$l_{P}(A_{1}'')<l_{P}({A_{1}'})$ by the first variation formula,
see
\cite[Lemma~3.5.1]{VToponogov:05},
\begin{equation}\label{compute partialgen}
 \frac{d}{ds}\,l_{Q}(s)=\cos\varphi^i_{c}(s)+\cos\alpha_{i}(s).
\end{equation}
Therefore, $\cos\alpha_{i}(s)=0$ and
\begin{equation*}
 \frac{d}{ds}\,l_{Q}(s)=\cos\varphi^i_{c}(s).
\end{equation*}

If $Q=A_{1}^{\prime}$ then we have the line segment
$c_{A_{1}'P}(s)$ that connects the points $A_{1}^{\prime}$ and
$P$, and  $c_{A_{2}'}(s,t)$ is a line that passes through
$A_{2}^{\prime}$ and $c_{A_{3}'}(s,t)$ is a line that passes
through $A_{3}^{\prime}$. Therefore,
\begin{equation*}
 \frac{d}{ds}\,l_{A_{1}^{\prime}}(s)=\cos\varphi_{A_{3}^{\prime}}(s),\quad
 \frac{d}{ds}\,l_{A_{1}^{\prime}}(s)=\cos\varphi_{A_{2}^{\prime}}(s),
\end{equation*}
where $c(s)$ counts from the point $A_{2}^{\prime}$ to $P$, and
from $A_{3}^{\prime}$ to $P$, respectively.

Similarly, for $Q=A_{2}^{\prime}$ and $Q=A_{3}^{\prime}$, we
obtain
\begin{eqnarray}\label{cycleB1}
 \frac{d}{ds}\,l_{A_{2}^{\prime}}(s)=\cos\varphi_{A_{3}^{\prime}}(s),\quad
 \frac{d}{ds}\,l_{A_{2}^{\prime}}(s)=\cos\varphi_{A_{1}^{\prime}}(s),\\
\label{cycleC1}
 \frac{d}{ds}\,l_{A_{3}^{\prime}}(s)=\cos\varphi_{A_{2}^{\prime}}(s),\quad
 \frac{d}{ds}\,l_{A_{3}^{\prime}}(s)=\cos\varphi_{A_{1}^{\prime}}(s).
\end{eqnarray}
 Since $t$ is a canonical parameter, we choose the parametrization
\[
 l_{A_{1}}(s)=\int_{0}^{s}\Big\|\frac{d}{dt}\,\gamma_{A_{1}^{\prime}P}(t)\Big\|\,{\rm d}t =\int_{0}^{s}\Big\|\frac{d}{dt}\,\gamma_{A_{1}^{\prime}}(s,t)\Big\|\,{\rm d}t=s,
\]
that is
\begin{equation}\label{compute partialparameter}
 l_{A_{1}^{\prime}}(s)=l_{A_{1}^{\prime}}=s.
\end{equation}
We assume that the distances $l_{A_{2}^{\prime}}$,
$l_{A_{3}^{\prime}}$ can be expressed as functions of
$l_{A_{1}^{\prime}}$,
\begin{equation*}
 l_{A_{2}^{\prime}}=l_{A_{2}^{\prime}}(l_{A_{1}^{\prime}}),\qquad
 l_{A_{3}^{\prime}}=l_{A_{3}^{\prime}}(l_{A_{1}^{\prime}}).
\end{equation*}
From this
and (\ref{minimum}) the following equation is obtained:
\begin{equation*}
 w_{1}l_{A_{1}^{\prime}}+w_{2}l_{A_{2}^{\prime}}(l_{A_{1}^{\prime}})+w_{3}l_{A_{3}^{\prime}}(l_{A_{1}^{\prime}})
 \to\min.
\end{equation*}
Differentiating this
with respect to the variable $l_{A_{1}^{\prime}}$ and using
(\ref{compute partialparameter}), we get
\begin{equation}\label{eq:B221frac}
 w_{1}+w_{2}\frac{d\,l_{A_{2}^{\prime}}}{d\,l_{A_{1}^{\prime}}}
 +w_{3}\frac{d\,l_{A_{3}^{\prime}}}{d\,l_{A_{1}^{\prime}}}=0.
\end{equation}
From (\ref{cycleB1})
and (\ref{cycleC1})
we get
\begin{equation}\label{partialimp2}
 \frac{d\,l_{A_{2}^{\prime}}}{d\,l_{A_{1}^{\prime}}}=\cos\varphi_{A_{3}^{\prime}}(s),\qquad
 \frac{d\,l_{A_{3}^{\prime}}}{d\,l_{A_{1}^{\prime}}}=\cos\varphi_{A_{2}^{\prime}}(s).
\end{equation}
Replacing (\ref{partialimp2}) and (\ref{compute partialparameter})
in (\ref{eq:B221frac}), we obtain
\begin{equation}\label{equation1}
 w_{1}+w_{2}\cos\varphi_{A_{3}^{\prime}}(l_{A_{1}^{\prime}})+w_{3}\cos\varphi_{A_{2}^{\prime}}(l_{A_{1}^{\prime}})=0.
\end{equation}
Similarly, working cyclically, we choose the parametrization
\[
 l_{A_{2}'}(s')=\int_{0}^{s'}\Big\|\frac{d}{dt}\,\gamma_{A_{2}'}(s',t)\Big\|\,{\rm d}t=s',\quad
 l_{A_{3}'}(s'')=\int_{0}^{s''}\Big\|\frac{d}{dt}\,\gamma_{A_{3}'}(s'',t)\Big\|\,{\rm d}t=s''.
\]
Differentiating (\ref{minimum}) with respect to $l_{A_{2}}$ for
$s^{\prime}=l_{A_{2}^{\prime}}$ and $l_{A_{3}^{\prime}}$ for
$s^{\prime\prime}=l_{A_{3}^{\prime}}$, we get
\begin{eqnarray}
\nonumber
 w_{1}\cos\varphi_{A_{3}'}(l_{A_{2}'})+w_{2}+w_{3}\cos\varphi_{A_{1}'}(l_{A_{2}'})=0,\\
\label{equation3}
 w_{1}\cos\varphi_{A_{2}'}(l_{A_{3}'})+w_{2}\cos\varphi_{A_{1}'}(l_{A_{3}'})+w_{3}=0.
\end{eqnarray}
From the uniqueness of the generalized F-T point $P$, we obtain
\[
 \varphi_{A'_{Q}}(l_{Q})=\varphi_{A'_{Q}}(l_{R})=\varphi_{A'_{Q}}(l_{S})=\varphi_{A'_{Q}},
\]
and that the solution of the linear system
(\ref{equation1})\,--\,(\ref{equation3}) is (\ref{E-wi}).
\qed

\begin{remark}\label{cocircles}\rm
a) Setting $w_{1}=w_{2}=w_{3}$ in (\ref{E-wi}), we get
$\varphi_{A'_Q}=120^{o}$. Hence, if the generalized equally
weighted F-T point $P$ is an interior point of $\Omega$ for $n=3$
then $\varphi_{A'_Q}=120^{o}$ (Isogonal property of the
generalized F-T point for equal weights).

b) If each closed convex curve $\gamma_{i}\ (i=1,2,3)$ approaches
to the circle $C(A_i,r_i)$ with the same perimeter then the
limiting case of Problem~\ref{probr2} for $n=3$ is the F-T problem
in $\mathbb{R}^{2}$ and $P\to P_{F}$, where $P_{F}$ is the F-T
point of $\triangle A_1 A_2 A_3$.
\end{remark}

\section{The generalized inverse weighted F-T problem for closed convex sets in $\mathbb{R}^2$}
\label{sec:inverse}

In the section we generalize results of
\cite{Zachos:2013a,Zachos:2013b}, where $\gamma_i=A_i$.
We shall
give the definition of dynamic plasticity for $n$ non-overlapping
convex sets in~$\mathbb{R}^2$.

\begin{definition}\label{dynamicplasticity}\rm
We call \textit{dynamic plasticity of $\,n$ non-overlapping convex
sets} $C_{i}$ in $\mathbb{R}^{2}$ the set of solutions
$\{(w_{1})_{1\ldots n},\ldots,(w_{n})_{1\ldots n}\}$ of
Problem~\ref{Problem-2} for $n$ sets $C_{i}$ with corresponding
variable weights $(w_{i})_{1\ldots n}$.
\end{definition}




\begin{proposition}[see \cite{Zach/Zou:08,Zachos/Cots:10}]\label{propo5}
Given the generalized F-T point $P$ to be an interior point of
$\Omega$ with the vertices defined by the three orthogonal
projections $A_i'$ of $P$ onto $\gamma_{i}\ (i=1,2,3)$ lie on
three line segments $PA_i'$ and form the given angles
$\varphi_{i}$, the positive weights $w_{i}$ are the solution of
Problem~\ref{Problem-2}:
\begin{equation*}
 w_{Q}=
 \Big(1+\frac{\sin{\varphi_{R}}}{\sin{\varphi_{Q}}}+\frac{\sin{\varphi_{S}}}{\sin{\varphi_{Q}}}\Big)^{-1}
 \quad{\rm for} \ \  Q,R,S\in\{1,2,3\}\ \ {\rm and} \ \ Q\ne R\ne S.
\end{equation*}
\end{proposition}



Let $(w_i)_{1234}$ be the weight corresponding to the point $A_i'$
of the closed convex curve $\gamma_{i}$, which is a vertex of the
convex quadrilateral $A_{1}'A_{2}'A_{3}'A_{4}'$, and let
$(w_j)_{jkl}$ be the weight corresponding to the point $A_j$ of
the closed convex curve $\gamma_{j}$, which is a vertex of
$\triangle A_jA_kA_l\ (j,k,l=1,\ldots,4)$.
 Furthermore, assume that $P$ lies at the interior of $\triangle A_1'A_2'A_3'\cap\triangle A_1'A_2'A_4'$
 and at the exterior of $\triangle A_1'A_3'A_4'.$

The following theorem deals with equations of dynamic plasticity
with respect to four closed convex sets in $\mathbb{R}^{2}$ and
their corresponding weights.

\begin{theorem}[The dynamic plasticity of four non-overlapping closed convex sets]\label{T-01}
Consider the Problem~\ref{Problem-2} for $n=4$. The dynamic
plasticity of four non-overlapping closed convex sets
$($e.g.circles$)$ $C_{i}\ (i=1,\ldots,4)$ and their corresponding
weights $(w_{i})_{1\ldots 4}$ is given by the following three
equations:
\begin{eqnarray}\label{plastic1g}
 \Big(\frac{w_2}{w_1}\Big)_{1\ldots4}=\Big(\frac{w_2}{w_1}\Big)_{123}
 \Big[1-\Big(\frac{w_4}{w_1}\Big)_{1\ldots4}\Big(\frac{w_1}{w_4}\Big)_{134}\,\Big],\\
\label{plastic2q}
 \Big(\frac{w_3}{w_1}\Big)_{1\ldots4}=\Big(\frac{w_3}{w_1}\Big)_{123}
 \Big[1-\Big(\frac{w_4}{w_1}\Big)_{1\ldots4}\Big(\frac{w_1}{w_4}\Big)_{124}\,\Big],\\
 \label{plastic3q}
 \sum\nolimits_{i=1}^{4}\big({w_i}\big)_{1\ldots4}=\const.
\end{eqnarray}
\end{theorem}

\textbf{Proof}. By applying the variational method that was used
in the proof of Theorem~\ref{T-sol2}, we obtain the weighted
'cosine' equations:
\begin{eqnarray}\label{equation1bis}
 w_{1}+w_{2}\cos\angle\,A_{1}'PA_{2}'+w_{3}\cos\angle\,A_{1}'PA_{3}' +w_{4}\cos\angle\,A_{1}'PA_{4}'=0,\\
\label{equation2bis}
 w_{1}\cos\angle\,A_{2}'PA_{1}' +w_{2}+w_{3}\cos\angle\,A_{2}'PA_{3}' +w_{4}\cos\angle\,A_{2}'PA_{4}'=0,\\
\label{equation3bis} w_{1}\cos\angle\,A_{3}'PA_{1}'
+w_{2}\cos\angle\,A_{3}'PA_{2}'
+w_{3}+w_{4}\cos\angle\,A_{3}'PA_{4}'=0.
\end{eqnarray}
Solving (\ref{equation1bis}) and (\ref{equation2bis}) with respect
to $w_{1}$ and $w_{2}$, we derive the weighted 'sine' equations:
\begin{eqnarray}\label{equation3biss}
\nonumber
 -w_{1}\sin\angle\,A_{2}'PA_{1}' +w_{3}\sin\angle\,A_{2}'PA_{3}' +w_{4}\sin\angle\,A_{2}'PA_{4}'=0,\\
 -w_{2}\sin\angle\,A_{1}'PA_{2}' +w_{3}\sin\angle\,A_{1}'PA_{3}' +w_{4}\sin\angle\,A_{1}'PA_{4}'=0.
\end{eqnarray}
Solving (\ref{equation1bis}) and (\ref{equation3bis}) with respect
to $w_{1}$ and $w_{3}$, we derive the
weighted 'sine' equation:
\begin{eqnarray}\label{equation4biss}
 -w_{1}\sin\angle\,A_{3}'PA_{1}' +w_{2}\sin\angle\,A_{3}'PA_{2}' -w_{4}\sin\angle\,A_{3}'PA_{4}'=0.
\end{eqnarray}
Applying Proposition~\ref{propo5} with respect to
$\triangle A_1'A_2'A_3'$, $\triangle A_1'A_2'A_4'$ and $\triangle
A_1'A_3'A_4'^*$, where $A_4'^*$ is the symmetric point of $A_4'$
with respect to $P$, we get
\begin{equation*}
 \Big(\frac{w_2}{w_1}\Big)_{123}\!=\frac{\sin\angle\,A_{1}'PA_{3}'}{\sin{\angle\,A_{2}'PA_{3}'}},\ \
 \Big(\frac{w_1}{w_4}\Big)_{124}\!=\frac{\sin\angle\,A_{2}'PA_{4}'}{\sin{\angle\,A_{1}'PA_{2}'}},\ \
 -\Big(\frac{w_1}{w_4}\Big)_{134}\!=\frac{\sin{\angle\,A_{3}'PA_{4}'}}{\sin\angle\,A_{1}'PA_{3}'}.
\end{equation*}
By replacing these solutions of Proposition~\ref{propo5} in
(\ref{equation4biss}) and (\ref{equation3biss})$_1$, we get
(\ref{plastic1g}) and (\ref{plastic2q}), respectively. \qed

\begin{remark}\rm
A similar system of weighted 'sine' equations has been obtained in
\cite[Theorem~1]{Zachos:2013b}, where four closed convex sets
degenerate to four fixed points.
\end{remark}

\begin{theorem}\label{gplasticityprincipleR2}
Given four rays which meet at the generalized F-T point $P$ and
their orthogonal projections with respect to four closed convex
curves $\gamma_{i}\ (i=1,\ldots,4)$ form a convex quadrilateral in
$\mathbb{R}^{2}$, an increase of the weight that corresponds to a
ray causes a decrease to the two weights that correspond to the
two neighboring rays and an increase to the weight that
corresponds to the opposite ray.
\end{theorem}

\textbf{Proof}. Taking into account that the four rays which meet
at the generalized F-T point $P$ intersects each $\gamma_i$ at a
right angle, we derive that the first variation formula of the
length of a ray with respect to arc length (see
\cite[Lemma~3.5.1]{VToponogov:05}) coincides with the first
variation formula of the length of the ray with respect to arc
length from four fixed points, respectively, and thus obtain the
same angular relation. By~following the process that was used in
\cite{Zachos:2013b}, we derive the same plasticity equations of
Theorem~\ref{T-01}. Assuming
\begin{eqnarray*}
 \sum\nolimits_{1234}w=\sum\nolimits_{123}w=\sum\nolimits_{124}w=\sum\nolimits_{134}w=\sum\nolimits_{234}w\,,
\end{eqnarray*}
by Theorem~\ref{T-01} we obtain
 $\,(w_i)_{1234}=a_i (w_4)_{1234}+ (w_i)_{123}$ for $i=1,2,3$,
where
\begin{eqnarray*}
 a_1 \eq \frac{\big(\frac{w_1}{w_4}\big)_{134}\big(\frac{w_2}{w_1}\big)_{123}
 +\big(\frac{w_1}{w_4}\big)_{124}\big(\frac{w_3}{w_1}\big)_{123}-1}
 {1+\big(\frac{w_2}{w_1}\big)_{123}+\big(\frac{w_3}{w_1}\big)_{123}},\\
 a_2\eq a_{1}\Big(\frac{w_2}{w_1}\Big)_{123}-\Big(\frac{w_1}{w_4}\Big)_{134}\Big(\frac{w_2}{w_1}\Big)_{123},\\
 a_3 \eq a_{1}\Big(\frac{w_3}{w_1}\Big)_{123}-\Big(\frac{w_1}{w_4}\Big)_{124}\Big(\frac{w_3}{w_1}\Big)_{123}.
\end{eqnarray*}
From this the claim follows. \qed

\smallskip

We will discuss a connection of the generalized F-T problem for
$n$ circles $C(A_{i},r_{i})$ in $\mathbb{R}^{2}$ and a degenerate
Steiner problem for $n$ circles $C(A_{i},r_{i})$ and one mobile
vertex $P$ (or circle $C(P,r)$) at the convex hull of
$\{A_{1},\cdots, A_{n}\}$. The~Steiner problem states:

\begin{problem}
\rm
Find all
networks of minimal length spanning points $\{A_{1},\cdots,
A_{n}\}$ in~$\mathbb{R}^{2}$.
\end{problem}

A kind of this problem (called the degenerate Steiner problem)
states:


\begin{definition}\rm
A \textit{generalized F-T tree} for $n$ given circles $C_{i}$ in
$\mathbb{R}^2$ is the solution of the following problem: ``Let
$\{C_{1},C_{2},\cdots C_{n}\}$ be $n$ given non-overlapping
circles and a positive real number (weight) corresponds to each
circle $C_{i}$ in $\mathbb{R}^{2}$ and one mobile circle $C(P,r)$
located at the convex hull of $\{A_{1},A_{2},\cdots A_{n}\}$.
Describe all the minimal weighted networks (networks of minimal
length) spanning $\{C_{1},C_{2},\cdots C_{n}\}$".
\end{definition}

The unique generalized F-T tree with respect to circles
$\{C_{i}\}$ consists of $n$ (weigh\-ted) line segments
$A_{i}^{\prime}P$ (branches) which intersect at the generalized
F-T point $P,$ where $A_{i}^{\prime}$ is the intersection point of
$C_{i}$ with the segment $A_{i}P$.
The geometric plasticity of weighted F-T tree networks for
quadrilaterals on surfaces was defined in \cite{Zachos:2013b}.
We shall extend this definition regarding the generalized F-T tree
problem for $n$ circles $C_{i}$.
The generalized F-T tree solution may also be viewed as a
branching solution and the generalized F-T point $P$ -- as the
branching point.

\begin{definition}\label{geometric plasticity}\rm
We call \textit{geometric plasticity of a weighted generalized F-T
tree} a network through $n$ non-overlapping circles $C_{i}$ for
given weights $(w_{i})_{1\cdots n}$, which correspond to each
circle the set of branching solutions of $n$ variable branches
$A_{i}^{\prime}P$ that preserve the generalized F-T point $P$ at
the same location of the $n$-gons formed by the endpoints of
$A_{i}^{\prime}P.$
\end{definition}

The field of branching solutions to various one-dimensional
variational problems has been introduced in \cite{Ivan/Tuzh:01}.
 Note that the geometric plasticity of a weighted generali\-zed
F-T tree for $n$ non-overlapping circles $C_{i}(A_{i},r_{i})$
permits a parallel translation of the circles in direction of rays
defined by $A_{i}P$.

We generalize the geometric plasticity principle derived in
\cite[Theorem~3, Quad\-rilaterals]{Zachos:2013b} for $n$ given
circles $C_{i}$ in $\mathbb{R}^{2}$.

\begin{theorem}\label{geometric plasticity principle of n circles}
Let $C_{i}(A_{i},r_{i})$ be $n$ non-overlapping circles such that
$A_1\cdots A_{n}$ is an $n$-gon in $\mathbb{R}^{2}$ and each
vertex $A_i$ possesses a non-negative weight $w_i$ for
$i=1,\cdots,n$. Assume that the floating case of the generalized
F-T problem is valid:
\[
 \big\|{\sum\nolimits_{j\neq i} w_{j}\,\vec u(A_j,A_i)}\big\| > w_i,\quad i\in\{1,\cdots,n\},
\]
and select $r_{i}\ (i=1,\cdots,n)$ such that $P$ does not belong
to the disk bounded by $C_{i}(A_{i},r_{i})$. Assume that $P$ is
connected with every vertex $A_i$ for $i=1,\cdots,n$, a circle
$C_{i}^{\prime}$ with center $A_i'$ has a non-negative weight
$w_i$ at the line that is defined by the line segment $PA_i$, an
$n$-gon $A_1'\dots A_n'$ has the property
\[
 \big\|{\sum\nolimits_{j\neq i} w_{j}\vec u(A_j',A_i')}\big\|>w_i,\quad i=1,\cdots,n,
\]
and $P$ does not belong to the disk $C_{i}(A_{i}',r_{i}')$ for
$i=1,\cdots,n$. Then the generalized  F-T point $P^{\prime}$
equals to $P$ $($the geometric plasticity principle$)$.
\end{theorem}

\begin{proof}
$P$ is the generalized F-T floating point of circles $C_{1}\ldots
C_{n}$ whose centers form a convex polygon $A_{1}\ldots A_{n}$ in
$\mathbb{R}^{2}$. Thus, the weighted floating equilibrium
condition holds:
\[
 \sum\nolimits_{\,i=1}^{n}w_{i}\vec u(P,A_i)={0}.
\]
If $P'$ is the generalized F-T floating point of the circles
$C'_{1}\ldots C'_{n}$ then
\[
 \sum\nolimits_{\,i=1}^{n} w_{i}\vec u(P',A'_i)={0};
\]
hence, $P\equiv P'$ because $A'_{i}$ is located at the ray defined
by $PA_{i}$ starting from~$P$.
\end{proof}

Therefore, the plasticity of a generalized F-T tree of $n$ circles
deals with the simultaneous occurence of both the dynamic
plasticity and the geometric plasticity of the corresponding
variable weighted tree network.

\begin{proposition}\label{theorP5}
Consider Problem~\ref{Problem-2}
with respect to $n$ circles $C(A_{i},r_{i})$.
The~following equations point out the \underline{plasticity} of
the system:
\begin{equation}\label{plastic1P5}
\begin{array}{c}
 (\frac{w_2}{w_1})_{1\dots n}=(\frac{w_2}{w_1})_{123}\,
 [1
 -(\frac{w_4}{w_1})_{1\dots n}(\frac{w_1}{w_4})_{134}
 -\ldots
 -(\frac{w_n}{w_1})_{1\dots n}(\frac{w_1}{w_n})_{13n}],\\
 \\
 (\frac{w_3}{w_1})_{1\dots n}=(\frac{w_3}{w_1})_{123}\,
 [1
 -(\frac{w_4}{w_1})_{1\dots n}(\frac{w_1}{w_4})_{124}
 -\ldots
 -(\frac{w_n}{w_1})_{1\dots n}(\frac{w_1}{w_n})_{12n}].\
\end{array}
\end{equation}
The weight $(w_i)_{1\dots n}$ corresponds to the vertex $A_i$ of
$A_{1}\ldots A_{n}$, and the weight $(w_j)_{jkl}$ corresponds to
the
vertex $A_j$ of
$\triangle A_jA_kA_l\ (j,k,l=1,\dots,n)$.
\end{proposition}

\textbf{Proof}. We follow the process used in
\cite[Proposition~4.4]{Zachos/Zou:88} for a convex $n$-gon, and
assume that $n-3$ branches grow simultaneously from the point $P$
and belong to $\angle\,A_{1}PA_{4}$ such that the vector
$\overline{P A_{i}}$ belongs to $\angle\,A_{1}PA_{i-1}$ for
$i=5,\dots,n$. Applying the cosine law to $\triangle\,PA_{i}A_{3}$
for $i\in\{1,\dots,n\}$ and $i\neq 3$, allows us to consider the
distance ${d}(P,A_{i})$ as a function of two variables,
${d}(P,A_{3})$ and $\angle\,PA_{3}A_{2}$\,:
\begin{equation*}
 {d}(P,A_{i})^{2} \!= {d}(P,A_{3})^{2}\!+{d}(A_{3},A_{i})^{2}
 \!-2{d}(P,A_{3}){d}(A_{3},A_{i})\cos(\angle A_{2}A_{3}A_{i}-\angle PA_{3}A_{2}).
\end{equation*}
Differentiating the objective function
$f(P)=\sum\nolimits_{i}w_{i}{d}(P,A_{i})$ with respect to the
variable $\angle\,PA_{3}A_{2}$ (see also the differentiation in
Corollary~\ref{cocircles}),
we obtain
\begin{equation*}
 w_{1}\sin\angle\,A_{1}P A_{3}-w_{2}\sin\angle A_{2}PA_{3}+w_{4}\sin\angle A_{3}P A_{4}+\ldots
 +w_{n}\sin\angle A_{3}P A_{n}=0.
\end{equation*}
From this
we have
\begin{equation*}
 \Big(\frac{w_2}{w_1}\Big)_{1\ldots n}\!=\frac{\sin\angle A_{3}PA_{1}}{\sin\angle A_{2}PA_{3}} \Big[1+\Big(\frac{w_4}{w_1}\Big)_{1\ldots n}\frac{\sin\angle A_{3}PA_{4}}{\sin\angle A_{3}PA_{1}} +\ldots+\Big(\frac{w_n}{w_1}\Big)_{1\ldots n}\frac{\sin\angle A_{3}PA_{n}}{\sin\angle A_{3}PA_{1}}\Big].
\end{equation*}
Taking into account Proposition~\ref{propo5} with respect to
$\triangle A_1A_2A_3$ and $\triangle A_1A_3A_i^*$, where $A_i^*$
is the symmetric point of $A_{i}$ with respect to $P$ for
$i=4,\dots,n$, we get
\begin{equation*}
 \Big(\frac{w_2}{w_1}\Big)_{123}=\frac{\sin\angle\,A_{1}PA_{3}}{\sin{\angle\,A_{2}PA_{3}}},\quad
 -\Big(\frac{w_1}{w_n}\Big)_{13n}=\frac{\sin{\angle\,A_{3}PA_{n}}}{\sin\angle\,A_{1}PA_{3}}.
\end{equation*}
Similarly, differentiating the objective function
$f(P)=\sum\nolimits_{i}w_{i}{d}(P,A_{i})$ with respect to the
variable $\angle\,PA_{2}A_{3}$, we obtain
\begin{equation*}
 -w_{1}\sin\angle\,A_{1}P A_{2}+w_{3}\sin\angle\,A_{2}PA_{3}
 +\ldots
 +w_{n}\sin\angle\,A_{2}P A_{n}=0.
\end{equation*}
Taking into account Proposition~\ref{propo5} with respect to
$\triangle A_1A_2A_i$ for $i=3,\dots,n$, we~get
\[
 \Big(\frac{w_3}{w_1}\Big)_{123}=\frac{\sin\angle\,A_{1}P A_{2}}{\sin{\angle\,A_{2}P A_{3}}},\quad
 \Big(\frac{w_1}{w_i}\Big)_{12i}=\frac{\sin\angle\,A_{2}P A_{i}}{\sin\angle\,A_{1}P A_{2}}.\qquad\qed
\]

\begin{remark}\rm
1. Choosing a proper orientation of angles that was used in
\cite{Zachos:2013b}, we also derive the plasticity equations
(\ref{plastic1P5}) applying the first variation formula of line
segments with respect to arc length and the parametrization that
was used in \cite{Zachos/Cots:10}.

2. The dynamic plasticity of the degenerate case of quadrilaterals
and convex pentagons as a limiting case of closed hexahedra is
studied in \cite[Theorem~1 and Corollary~3]{Zachos:2013b}
and \cite[Theorem~6]{Zachos:2013a}, respectively.
\end{remark}

Denote by
\[
 \sum\nolimits_{1\ldots n}w:=\sum\nolimits_{i=1}^n \Big(\frac{w_i}{w_1}\Big)_{1\ldots n}.
\]

\begin{corollary}\label{corolllngon}
Let
\[
 \sum\nolimits_{1\dots n}^{}w=\sum\nolimits_{123}^{}w=\sum\nolimits_{124}^{}w=\sum\nolimits_{134}^{}w
 =\ldots=\sum\nolimits_{1(n-1)n}^{}w\,.
\]
Then $(w_i)_{1\dots n}= a_{i,4}(w_{4})_{1\dots n}
+\dots+a_{i,n}(w_{n})_{1\dots n} +a_{i,n+1}$ for $i=1,2,3$, where
\begin{eqnarray*}
 &&\hskip-19mm (a_{1,4},\dots,a_{1,n},a_{1,n+1})=\Big[\,\frac{\big(\frac{w_1}{w_4}\big)_{134}\big(\frac{w_2}{w_1}\big)_{123}
 +\big(\frac{w_1}{w_4}\big)_{124}\big(\frac{w_3}{w_1}\big)_{123}-1}
 {1+\big(\frac{w_2}{w_1}\big)_{123}+\big(\frac{w_3}{w_1}\big)_{123}} , \\
  && \dots\,, \frac{\big(\frac{w_1}{w_n}\big)_{13n}\big(\frac{w_2}{w_1}\big)_{123}
  +\big(\frac{w_1}{w_n}\big)_{12n}\big(\frac{w_3}{w_1}\big)_{123}-1}
  {1+\big(\frac{w_2}{w_1}\big)_{123}+\big(\frac{w_3}{w_1}\big)_{123}},\, (w_1)_{123}\Big]\\
  &&\hskip-19mm  (a_{2,4},\dots,a_{2,n},a_{2,n+1})
  =\Big[\,a_{1,4}\Big(\frac{w_2}{w_1}\Big)_{123}\!
  -\Big(\frac{w_1}{w_4}\Big)_{134}\Big(\frac{w_2}{w_1}\Big)_{123}, \\
  &&
  \dots,a_{1,n}
 \Big(\frac{w_2}{w_1}\Big)_{123}-\Big(\frac{w_1}{w_n}\Big)_{13n}\Big(\frac{w_2}{w_1}\Big)_{123},\, (w_2)_{123}\Big]\\
 &&\hskip-19mm (a_{3,4},\dots,a_{3,n},a_{1,n+1})=\Big[\,a_{1,4}\Big(\frac{w_3}{w_1}\Big)_{123}-\Big(\frac{w_1}{w_4}\Big)_{124}
 \Big(\frac{w_3}{w_1}\Big)_{123},\\
  &&\dots\,, a_{1,n}\Big(\frac{w_3}{w_1}\Big)_{123}-\Big(\frac{w_1}{w_n}\Big)_{12n}\Big(\frac{w_3}{w_1}\Big)_{123},\,(w_3)_{123}\,\Big].
\end{eqnarray*}
\end{corollary}

\textbf{Proof}. This is a direct consequence of
(\ref{plastic1P5}): by taking into account that
$$
 \sum\nolimits_{1\dots n} w:=(w_1)_{1\dots n} \Big(1+\sum\nolimits_{j=2}^n\frac{w_j}{w_1}\Big)_{1\dots n}=\sum\nolimits_{ijk}{w}
$$
for $i,j,k=1,\dots,n$ and by solving with respect to
$(w_i)_{1\dots n}$ we derive the plasticity equations for
$i=1,2,3$ which depend on $n-3$ variables $(w_{i})_{1\dots n}$ for
$i=4,\dots,n$. \qed

\begin{remark}\rm
If $n-4$ weights are zero, the plasticity principle holds by
Theorem~\ref{T-01}. Thus, Corollary~\ref{corolllngon} for $n=4$ is
a direct consequence of Theorem~\ref{T-01}.

\end{remark}

\begin{remark}\rm
The condition
\[
 \sum\nolimits_{1\dots n}^{}w=\sum\nolimits_{123}^{}w=\sum\nolimits_{124}^{}w=\sum\nolimits_{134}^{}w
 =\ldots=\sum\nolimits_{1(n-1)n}^{}w
\]
can be interpreted as equality between isoperimetric conditions of
the weights $(w_{i})_{ijk}$, $(w_{j})_{ijk},$ $(w_{k})_{ijk}$ with
respect to the triangles $\triangle A_{i}A_{j}A_{k}.$ These
weights corresponds to the side lengths of the dual restricted
weighted F-T problem for these specific triangles by letting the
rest $n-3$ weights zero with the isoperimetric condition of
initial generalized F-T problem for $n$ given weights $w_{1\dots
n}$.
 The dual F-T problem is connected with the F-T
problem by exchanging the lengths of the F-T problem with weights
and the new weights become edge lengths of the dual F-T problem
for triangles.
\end{remark}

The next proposition provides a control to Problem~\ref{Problem-2}
and extends the result \cite[Proposition~19.6]{BolMa/So:99}.

\begin{proposition}[Absorbed case of Problem~\ref{probr2}]
\label{controlincreasebranches} Let $W_{P}\subset\mathbb{R}^{2}$
consists of even number of circles with $m=2k\ (k\ge 2)$, and
their centers $A_{1},A_{2},\dots,A_{2k-1}$ be distributed around
$A_{m}$ in such a way that the interior of
$\angle\,A_{i_{1}}A_{m}A_{i_{2}}$ contains at least one of the
vectors $-w_{j}\vec{u}(A_{m},A_{j})$, $j=1,\dots,m-1$. Then
$P=P_{m}$.
\end{proposition}

\textbf{Proof}.
We can set $\sum_{i=1}^{2k-1}w_{i}=1$, where $w_{i}$ is the weight
corresponding to the vertex $A_{i}\ (i=1,2,\dots,2k-1)$ and
$w_{2k}=1$ corres\-ponds to
$A_{2k}$.
 By the weighted absorbed case \cite[p.~250]{BolMa/So:99}, we need to get
 $\|\sum\nolimits_{i=1}^{2k-1}w_{i}\vec{u}(A_{m},A_{i})\|\le 1$,
where $\vec{u}(A_{m},A_{i})$ is a unit vector in the direction
of~$A_{m}A_{i}$.
Set $w_{i}\vec{u}(A_{m},A_{i})=\vec{u^{\prime}}(A_{m},A_{i})$.
Then the proof goes line by line in the same way as given in
\cite[Proposition~19.6]{BolMa/So:99}.
\qed

\section{Some new types of evolution of five non-overlapping circles in $\mathbb{R}^{2}$}
\label{sec:non-over}

We obtain two new types of evolutionary structures of five
non-overlapping circles with variable radius such that their
(fixed) centers form a convex pentagon by deriving the iteration
plasticity equations and their radius may be taken to be a scaling
of weights $w_{i}$. The evolution of five such circles is a
generalization of the evolution of weighted pentagons in
$\mathbb{R}^{2}$ which has been given in
\cite[Theorem~7]{Zachos:2013a}.

We may select a scaling constant with respect to
weights of
evolutio\-nary circles with variable radius, corresponding to
Problem~\ref{Problem-2} for $n=3$ not overlapping circles.

\begin{theorem}\label{typesplasticity}
The evolution of five non-overlapping circles, whose centers form
a convex pentagon, might be of the following two types:

Type A: For Problem~\ref{Problem-2} with $n=3$,
two branches grow simultaneously from $P$ inside
$\angle\,A_{1}PA_{3}$.

Type B: For Problem~\ref{Problem-2} with $n=3$,
one branch grows from $P$ inside $\angle\,A_{1}PA_{3}$, the other
branch grows inside $\angle\,A_{1}PA_{2}$, but both branches do
not grow simultaneously.
\end{theorem}

\textbf{Proof}. The characterization of type A is given by
Corollary~\ref{corolllngon} for $n=5$. These iteration plasticity
equations give an increase of weights $(w_{4})_{1\ldots5}$,
$(w_{5})_{1\ldots5}$ and $(w_{2})_{1\ldots5}$, and a decrease of
weights $(w_{1})_{1\ldots5}$ and $(w_{3})_{1\ldots5}$.

One may characterize type B by extracting Problem~\ref{Problem-2}
for $n=3$ from Problem~\ref{Problem-2} for four evolutionary
non-overlapping circles and derive the plasticity of new problem.
Thus, from the plasticity principle of four
non-overlapping circles the weights (radii) $(w_{4})_{1\ldots4}$
and $(w_{2})_{1\ldots4}$ increase, and the weights (radii)
$(w_{1})_{1\ldots4}$ and $(w_{3})_{1\ldots4}$ decrease. Therefore,
the fourth and second circles are decreased and the first and
third circle are increased, keeping their centers at vertices of a
fixed convex quadrilateral, because their radii are taken to be
scalings of weights $w_{i}\ (i=1,\ldots,4)$.

Composing the vectors $w_{4}\vec{u}(P,A_{4})$ and
$w_{3}\vec{u}(P,A_{3})$ we get $w_{3=4}\vec{u}(P,A_{3=4})$.
The\-refore, modified Problem~\ref{Problem-2} for three new
evolutionary non-overlapping circles depends on ratios of weights
$w_{1}$ $w_{2}$ and $w_{3=4}$. Taking into account the plasticity
principle of four non-overlapping circles whose centers form a
convex quadrilateral $A_{1}A_{5}A_{2}A_{3=4}$, we find that the
weights $(w_{5})_{152(3=4)}$ and $(w_{34})_{152(3=4)}$ increase,
while $(w_{1})_{152(3=4)}$ and $(w_{2})_{152(3=4)}$ decrease.
 Indices in parentheses mean possible reduction of pentagon to
quadrilateral (composing vectors of third and fourth rays)
starting from $\triangle A_{1}A_{2}A_{5}$. So, the weight
$w_{125(3=4)}$ corresponds to the reduced weight of the
evolutionary quadrilateral $A_{1}A_{2}A_{5}A_{3=4}$. \qed

\begin{remark}\rm
Following Theorem~\ref{typesplasticity}, we can derive various
types of evolution of $n$ non-overlapping circles with respect to
growing of branches from $P$. The termination of evolution may be
studied when the closed convex sets are
circles, whose radii may be taken to be a scaling of weights
$w_{i}$ from Corollary~\ref{corolllngon}, such that at every
circle might appear a leaf of tree-like type (see~\cite{XiaQ:07}
for formation~of~leaves and ~\cite{Zachos:2012,Zachos:2015} for
creation of tree leaf types).

Note that the algorithm and equations used for modeling the
formation of leaves in \cite{XiaQ:07} did not take into account
the main branch of the leaf and used a selection principle to
control solar energy. We may overpass the obstacle of creating a
monster leaf by controlling the solar energy using conditions for
the weights referring as isoperimetric conditions of the weights.
Concerning our model $n$ leaves may be derived by the generalized
F-T problem for $n$ circles (or closed convex sets) and $2k$
branches which could be moved via parallel translation inside
these circles. We call a leaf a circle (or closed convex set)
enriched with the structure of a finite number of parallel
translated $V$ shaped branches inside the circle along its main
branch.
\end{remark}

\section*{Acknowledgment}
The author would like to thank Prof. Dr. V.~Rovenski for his
comments and for many useful discussions and Prof. Q.~Xia for
communicating \cite{XiaQ:07}.

\end{document}